\documentclass[10pt]{amsart}
\usepackage{amssymb,MnSymbol}
\usepackage{amsthm,amsmath}
\usepackage{cite}

\title{Preassociative aggregation functions}

\author{Jean-Luc Marichal}
\address{Mathematics Research Unit, FSTC, University of Luxembourg \\
6, rue Coudenhove-Kalergi, L-1359 Luxembourg, Luxembourg} \email{jean-luc.marichal[at]uni.lu }

\author{Bruno Teheux}
\address{Mathematics Research Unit, FSTC, University of Luxembourg \\
6, rue Coudenhove-Kalergi, L-1359 Luxembourg, Luxembourg} \email{bruno.teheux[at]uni.lu }

\date{March 13, 2015}

\begin{document}

\theoremstyle{plain}
\newtheorem{theorem}{Theorem}[section]
\newtheorem{lemma}[theorem]{Lemma}
\newtheorem{proposition}[theorem]{Proposition}
\newtheorem{corollary}[theorem]{Corollary}
\newtheorem{fact}[theorem]{Fact}
\newtheorem*{main}{Main Theorem}

\theoremstyle{definition}
\newtheorem{definition}[theorem]{Definition}
\newtheorem{example}[theorem]{Example}

\theoremstyle{remark}
\newtheorem*{conjecture}{Conjecture}
\newtheorem{remark}{Remark}
\newtheorem{claim}{Claim}

\newcommand{\N}{\mathbb{N}}
\newcommand{\Q}{\mathbb{Q}}
\newcommand{\R}{\mathbb{R}}

\newcommand{\ran}{\mathrm{ran}}
\newcommand{\dom}{\mathrm{dom}}
\newcommand{\id}{\mathrm{id}}
\newcommand{\med}{\mathrm{med}}
\newcommand{\Ast}{\boldsymbol{\ast}}
\newcommand{\Cdot}{\boldsymbol{\cdot}}

\newcommand{\bfu}{\mathbf{u}}
\newcommand{\bfv}{\mathbf{v}}
\newcommand{\bfw}{\mathbf{w}}
\newcommand{\bfx}{\mathbf{x}}
\newcommand{\bfy}{\mathbf{y}}
\newcommand{\bfz}{\mathbf{z}}

\begin{abstract}
The classical property of associativity is very often considered in aggregation function theory and fuzzy logic. In this paper we provide axiomatizations of various classes of preassociative functions, where preassociativity is a generalization of associativity recently introduced by the authors. These axiomatizations are based on existing characterizations of some noteworthy classes of associative operations, such as the class of Acz\'elian semigroups and the class of t-norms.
\end{abstract}



\maketitle

\section{Introduction}

Let $X$ be an arbitrary nonempty set (e.g., a nontrivial real interval) and let $X^*=\bigcup_{n \geqslant 0} X^n$ be the set of all tuples on $X$, with the convention that $X^0=\{\varepsilon\}$ (i.e., $\varepsilon$ denotes the unique $0$-tuple on $X$). The \emph{length} $|\bfx|$ of a tuple $\bfx\in X^*$ is a nonnegative integer defined in the usual way: we have $|\bfx|=n$ if and only if $\bfx\in X^n$. In particular, we have $|\varepsilon|=0$.

In this paper we are interested in \emph{$n$-ary} functions $F\colon X^n\to Y$, where $n\geqslant 1$ is an integer, as well as in \emph{variadic} functions $F\colon X^* \to Y$, where $Y$ is a nonempty set. A variadic function $F\colon X^* \to Y$ is said to be \emph{standard} \cite{LehMarTeh} if the equality $F(\bfx)=F(\varepsilon)$ holds only if $\bfx =\varepsilon$. Finally, a variadic function $F\colon X^*\to X\cup\{\varepsilon\}$ is called a \emph{variadic operation on $X$} (or an \emph{operation} for short), and we say that such an operation is \emph{$\varepsilon$-preserving standard} (or \emph{$\varepsilon$-standard} for short) if it is standard and satisfies $F(\varepsilon)=\varepsilon$.

For any variadic function $F\colon X^*\to Y$ and any integer $n\geqslant 0$, we denote by $F_n$ the \emph{$n$-ary part} of $F$, i.e., the restriction $F|_{X^n}$ of $F$ to the set $X^n$. The restriction $F|_{X^*\setminus\{\varepsilon\}}$ of $F$ to the tuples of positive lengths is denoted $F^{\flat}$ and called the \emph{non-nullary part} of $F$. Finally, the value $F(\varepsilon)$ is called the \emph{default value} of $F$.

The classical concept of associativity for binary operations can be easily generalized to variadic operations in the following way. A variadic operation $F \colon X^* \to X\cup\{\varepsilon\}$ is said to be \textit{associative} \cite{LehMarTeh,MarTeh} (see also \cite[p.~24]{Mar98}) if
\begin{equation}\label{eq;assoc}
F(\bfx,\bfy,\bfz) = F(\bfx,F(\bfy),\bfz){\,},\qquad \bfx,\bfy,\bfz \in X^*.
\end{equation}
Here and throughout, for tuples $\bfx=(x_1,\ldots,x_n)$ and $\bfy=(y_1,\ldots,y_m)$ in $X^*$, the notation $F(\bfx,\bfy)$ stands for the function $F(x_1,\ldots,x_n,y_1,\ldots,y_m)$, and similarly for more than two tuples. We also assume that $F(\varepsilon,\bfx)=F(\bfx,\varepsilon)=F(\bfx)$ for every $\bfx\in X^*$.

Any associative operation $F \colon X^* \to X\cup\{\varepsilon\}$ clearly satisfies the condition $F(\varepsilon)=F(F(\varepsilon))$. From this observation it follows immediately that any associative standard operation $F \colon X^* \to X\cup\{\varepsilon\}$ is necessarily $\varepsilon$-standard.

Associative binary operations and associative variadic operations are widely investigated in aggregation function theory, mainly due to the many applications in fuzzy logic (for general background, see \cite{GraMarMesPap09}).


Associative $\varepsilon$-standard operations $F\colon X^*\to X\cup\{\varepsilon\}$ are closely related to associative binary operations $G\colon X^2\to X$, which are defined as the solutions of the functional equation
$$
G(G(x,y),z) ~=~ G(x,G(y,z)),\qquad x,y,z\in X.
$$
In fact, it can be easily seen \cite{MarTeh,MarTeh2} that a binary operation $G\colon X^2\to X$ is associative if and only if there exists an associative $\varepsilon$-standard operation $F\colon X^*\to X\cup\{\varepsilon\}$ such that $G=F_2$. Moreover, as observed in \cite[p.~25]{Mar98} (see also \cite[p.~15]{BelPraCal07} and \cite[p.~33]{GraMarMesPap09}), any associative $\varepsilon$-standard operation $F\colon X^*\to X\cup\{\varepsilon\}$ is completely determined by its unary and binary parts. Indeed, by associativity we have
\begin{equation}\label{eq:saf76sf5xx}
F_n(x_1,\ldots, x_n) ~=~ F_2(F_{n-1}(x_1,\ldots, x_{n-1}),x_n),\qquad n\geqslant 3,
\end{equation}
or equivalently,
\begin{equation}\label{eq:saf76sf5}
F_n(x_1,\ldots, x_n) ~=~ F_2(F_2(\ldots F_2(F_2(x_1,x_2),x_3),\ldots), x_n),\qquad n\geqslant 3.
\end{equation}


In this paper we are interested in the following generalization of associativity recently introduced by the authors in \cite{MarTeh,MarTeh2} (see also \cite{LehMarTeh}).

\begin{definition}[{\cite{MarTeh,MarTeh2}}]\label{de:7ads5sa}
A function $F\colon X^*\to Y$ is said to be \emph{preassociative} if for every $\bfx,\bfy,\bfy',\bfz\in X^*$ we have
$$
F(\bfy) ~=~ F(\bfy')\quad\Rightarrow\quad F(\bfx,\bfy,\bfz) ~=~ F(\bfx,\bfy',\bfz).
$$
\end{definition}

We can easily observe that any $\varepsilon$-standard operation $F\colon\R^*\to\R\cup\{\varepsilon\}$ defined by $F_n(\bfx)=f(\sum_{i=1}^nx_i)$ for every integer $n\geqslant 1$, where $f\colon\R\to\R$ is a one-to-one function, is an example of preassociative function.

It is immediate to see that any associative $\varepsilon$-standard operation $F\colon X^*\to X\cup\{\varepsilon\}$ necessarily satisfies the equation $F_1\circ F^{\flat}=F^{\flat}$ (take $\bfx=\bfz=\varepsilon$ in Eq.~(\ref{eq;assoc})) and it can be shown (Proposition~\ref{prop:A-PA1}) that an $\varepsilon$-standard operation $F\colon X^*\to X\cup\{\varepsilon\}$ is associative if and only if it is preassociative and satisfies $F_1\circ F^{\flat}=F^{\flat}$.

It is noteworthy that, contrary to associativity, preassociativity does not involve any composition of functions and hence allows us to consider a codomain $Y$ that may differ from $X\cup\{\varepsilon\}$. For instance, the length function $F\colon X^*\to\R$, defined by $F(\bfx)=|\bfx|$, is standard and preassociative.

In this paper we mainly consider preassociative standard functions $F\colon X^*\to Y$ for which $F_1$ and $F^{\flat}$ have the same range. (For $\varepsilon$-standard operations, the latter condition is an immediate consequence of the condition $F_1\circ F^{\flat}=F^{\flat}$ and hence these preassociative functions include all the associative $\varepsilon$-standard operations.) In Section 3 we recall the characterization of these functions as compositions of the form $F^{\flat}=f\circ H^{\flat}$, where $H\colon X^*\to X\cup\{\varepsilon\}$ is an associative $\varepsilon$-standard operation and $f\colon H(X^*\setminus\{\varepsilon\})\to Y$ is one-to-one.

In Section 4 we investigate the special case of standard functions whose unary parts are one-to-one. It turns out that this latter condition greatly simplifies the general results on associative and preassociative standard functions obtained in \cite{MarTeh,MarTeh2}. Section 5 contains the main results of this paper. We first recall axiomatizations of some noteworthy classes of associative $\varepsilon$-standard operations, such as the class of variadic extensions of Acz\'elian semigroups, the class of variadic extensions of t-norms and t-conorms, and the class of associative and range-idempotent $\varepsilon$-standard operations. Then we show how these axiomatizations can be extended to classes of preassociative standard functions. Finally, we address some open questions in Section 6.

Throughout the paper we make use of the following notation and terminology. We denote by $\N$ the set $\{1,2,3,\ldots\}$ of strictly positive integers. The domain and range of any function $f$ are denoted by $\dom(f)$ and $\ran(f)$, respectively. The identity operation on $X$ is the function $\id\colon X\to X$ defined by $\id(x) = x$.

\section{Preliminaries}

Recall that a function $F\colon X^n\to X$ ($n\in\N$) is said to be \emph{idempotent} (see, e.g., \cite{GraMarMesPap09}) if $F(x,\ldots,x)=x$ for every $x\in X$. Also, an $\varepsilon$-standard operation $F\colon X^*\to X\cup\{\varepsilon\}$ is said to be
\begin{itemize}
\item \emph{idempotent} if $F_n$ is idempotent for every $n\in\N$,

\item \emph{unarily idempotent} \cite{MarTeh,MarTeh2} if $F_1=\id$,

\item \emph{unarily range-idempotent} \cite{MarTeh,MarTeh2} if $F_1|_{\ran(F^{\flat})}=\id|_{\ran(F^{\flat})}$, or equivalently, $F_1\circ F^{\flat}=F^{\flat}$. In this case $F_1$ necessarily satisfies the equation $F_1\circ F_1=F_1$.
\end{itemize}

A function $F\colon X^*\to Y$ is said to be \emph{unarily quasi-range-idempotent} \cite{MarTeh,MarTeh2} if $\ran(F_1)=\ran(F^{\flat})$. Since this property is a consequence of the condition $F_1\circ F^{\flat}=F^{\flat}$ whenever $F$ is an $\varepsilon$-standard operation, we see that if an $\varepsilon$-standard operation $F\colon X^*\to X\cup\{\varepsilon\}$ is unarily range-idempotent, then it is necessarily unarily quasi-range-idempotent. The following proposition, stated in \cite{MarTeh} without proof, provides a finer result.

\begin{proposition}[{\cite{MarTeh,MarTeh2}}]\label{prop:22f1f1f1}
An $\varepsilon$-standard operation $F\colon X^*\to X\cup\{\varepsilon\}$ is unarily range-idempotent if and only if it is unarily quasi-range-idempotent and satisfies $F_1\circ F_1=F_1$.
\end{proposition}

\begin{proof}
(Necessity) We have $\ran(F_1)\subseteq\ran(F^{\flat})$ for any operation $F\colon X^*\to X\cup\{\varepsilon\}$. Since $F$ is unarily range-idempotent, we have $F_1\circ F^{\flat}=F^{\flat}$, from which the converse inclusion follows immediately. In particular, $F_1\circ F_1=F_1$.

(Sufficiency) Since $F$ is unarily quasi-range-idempotent, the identity $F_1\circ F_1=F_1$ is equivalent to $F_1\circ F^{\flat}=F^{\flat}$.
\end{proof}

Recall that a function $g$ is a \emph{quasi-inverse} \cite[Sect.~2.1]{SchSkl83} of a function $f$ if
\begin{eqnarray}
&& f\circ g|_{\ran(f)}=\id|_{\ran(f)},\label{eq:defqi1}\\
&& \ran(g|_{\ran(f)})=\ran(g).\label{eq:defqi2}
\end{eqnarray}

For any function $f$, denote by $Q(f)$ the set of its quasi-inverses. This set is nonempty whenever we assume the Axiom of Choice (AC), which is actually just another
form of the statement ``every function has a quasi-inverse.'' Recall also that the relation of being quasi-inverse is symmetric: if $g \in Q(f)$ then $f \in Q(g)$; moreover, we have $\ran(g)\subseteq\dom(f)$ and $\ran(f)\subseteq\dom(g)$ and the functions $f|_{\ran(g)}$ and $g|_{\ran(f)}$ are one-to-one.

By definition, if $g\in Q(f)$, then $g|_{\ran(f)}\in Q(f)$. Thus we can always restrict the domain of any quasi-inverse $g\in Q(f)$ to $\ran(f)$.
These ``restricted'' quasi-inverses, also called \emph{right-inverses} \cite[p.~25]{AlsFraSch06}, are then simply characterized by condition (\ref{eq:defqi1}), which can
be rewritten as
$$
g(y)\in f^{-1}\{y\},\qquad y\in\ran(f).
$$

The following proposition yields necessary and sufficient conditions for a function $F\colon X^*\to Y$ to be unarily quasi-range-idempotent.

\begin{proposition}[{\cite{MarTeh,MarTeh2}}]\label{prop:ACqriTFAE3451}
Assume AC and let $F\colon X^*\to Y$ be a function. The following assertions are equivalent.
\begin{enumerate}
\item[(i)] $F$ is unarily quasi-range-idempotent.

\item[(ii)] There exists an $\varepsilon$-standard operation $H\colon X^*\to X\cup\{\varepsilon\}$ such that $F^{\flat}=F_1\circ H^{\flat}$.

\item[(iii)] There exists a unarily idempotent $\varepsilon$-standard operation $H\colon X^*\to X\cup\{\varepsilon\}$ and a function $f\colon X\to Y$ such that $F^{\flat}=f\circ H^{\flat}$. In this case, $f=F_1$.
\end{enumerate}
In assertions (ii) we may choose $H^{\flat}=g\circ F^{\flat}$ for any $g\in Q(F_1)$ and $H$ is then unarily range-idempotent. In assertion (iii) we may choose $H_1=\id$ and $H_n=g\circ F_n$ for every $n>1$ and any $g\in Q(F_1)$.
\end{proposition}

We say that a function $F\colon X^*\to Y$ is \emph{unarily idempotizable} if it is unarily quasi-range-idempotent and $F_1$ is one-to-one. In this case the composition $F_1^{-1}\circ F^{\flat}$ from $X^*\setminus\{\varepsilon\}$ to $X$ is unarily idempotent. From Proposition~\ref{prop:ACqriTFAE3451}, we immediately derive the following corollary.

\begin{corollary}
Let $F\colon X^*\to Y$ be a function. The following assertions are equivalent.
\begin{enumerate}
\item[(i)] $F$ is unarily idempotizable.

\item[(ii)] $F_1$ is a bijection from $X$ onto $\ran(F^{\flat})$ and there is a unique unarily idempotent $\varepsilon$-standard operation $H\colon X^*\to X\cup\{\varepsilon\}$, namely $H^{\flat}=F_1^{-1}\circ F^{\flat}$, such that $F^{\flat}=F_1\circ H^{\flat}$.

\item[(iii)] There exist a unarily idempotent $\varepsilon$-standard operation $H\colon X^*\to X\cup\{\varepsilon\}$ and a bijection $f$ from $X$ onto $\ran(F^{\flat})$ such that $F^{\flat}=f\circ H^{\flat}$. In this case we have $f=F_1$ and $H^{\flat}=F_1^{-1}\circ F^{\flat}$.
\end{enumerate}
\end{corollary}

\section{Associative and preassociative functions}

In this section we recall some results on associative and preassociative variadic functions.

As the following proposition \cite{CouMar11,LehMarTeh} states, under the assumption that $F(\varepsilon)=\varepsilon$ there are different equivalent definitions of associativity (see also \cite[p.~32]{GraMarMesPap09}).

\begin{proposition}[{\cite{CouMar11,LehMarTeh}}]\label{prop:assocAltForms}
Let $F\colon X^*\to X\cup\{\varepsilon\}$ be an operation such that $F(\varepsilon)=\varepsilon$. The following assertions are equivalent:
\begin{enumerate}
\item[(i)] $F$ is associative.

\item[(ii)] For every $\bfx,\bfy,\bfz,\bfx',\bfy',\bfz'\in X^*$ such that $(\bfx,\bfy,\bfz)=(\bfx',\bfy',\bfz')$ we have $F(\bfx,F(\bfy),\bfz)=F(\bfx',F(\bfy'),\bfz')$.

\item[(iii)] For every $\bfx,\bfy\in X^*$ we have $F(\bfx,\bfy)=F(F(\bfx),F(\bfy))$.
\end{enumerate}
\end{proposition}

\begin{remark}
Associativity for $\varepsilon$-standard operations was defined in \cite{CouMar11} as in assertion (ii) of Proposition~\ref{prop:assocAltForms}. It was also defined in \cite[p.~16]{CalKolKomMes02}, \cite[p.~32]{GraMarMesPap09}, and \cite[p.~216]{KleMesPap20} as in assertion (iii) of Proposition~\ref{prop:assocAltForms}.
\end{remark}

Just as for associativity, preassociativity (see Definition~\ref{de:7ads5sa}) may have different equivalent forms. The following proposition, stated in \cite{MarTeh} without proof, gives an equivalent definition based on two equalities of values.

\begin{proposition}[{\cite{MarTeh,MarTeh2}}]\label{prop:4.1-78f6dg}
A function $F\colon X^*\to Y$ is preassociative if and only if for every $\bfx,\bfx',\bfy,\bfy'\in X^*$ we have
$$
F(\bfx) ~=~ F(\bfx')\quad\mbox{and}\quad F(\bfy) ~=~ F(\bfy')\quad\Rightarrow\quad F(\bfx,\bfy) ~=~ F(\bfx',\bfy').
$$
\end{proposition}

\begin{proof}
(Necessity) Let $\bfx,\bfy,\bfx',\bfy'\in X^*$. If $F(\bfx)=F(\bfx')$ and $F(\bfy)=F(\bfy')$, then we have $F(\bfx,\bfy)=F(\bfx',\bfy)=F(\bfx',\bfy')$.

(Sufficiency) Let $\bfx,\bfy,\bfy',\bfz\in X^*$. If $F(\bfy)=F(\bfy')$, then $F(\bfx,\bfy)=F(\bfx,\bfy')$ and finally $F(\bfx,\bfy,\bfz)=F(\bfx,\bfy',\bfz)$.
\end{proof}

As mentioned in the introduction, preassociativity generalizes associativity. Moreover, we have the following result.

\begin{proposition}[{\cite{MarTeh,MarTeh2}}]\label{prop:A-PA1}
An $\varepsilon$-standard operation $F\colon X^*\to X\cup\{\varepsilon\}$ is associative if and only if it is preassociative and unarily range-idempotent (i.e., $F_1\circ F^{\flat}=F^{\flat}$).
\end{proposition}

The following two straightforward propositions show how new preassociative functions can be generated from given preassociative functions by compositions with unary maps.

\begin{proposition}[Right composition]
If $F\colon X^*\to Y$ is standard and preassociative then, for every function $g\colon X'\to X$, the function $H\colon X'^*\to Y$, defined by $H_0=\mathbf{a}$ for some $\mathbf{a}\in Y\setminus\ran(F^{\flat})$ and $H_n=F_n\circ(g,\ldots,g)$ for every $n\in\N$, is standard and preassociative.
\end{proposition}

\begin{proposition}[Left composition]\label{prop:leftcomp56}
Let $F\colon X^*\to Y$ be a preassociative standard function and let $g\colon Y\to Y'$ be a function. If $g|_{\ran(F^{\flat})}$ is one-to-one, then the function $H\colon X^*\to Y$ defined by $H_0=\mathbf{a}$ for some $\mathbf{a}\in Y'\setminus\ran(g|_{\ran(F^{\flat})})$ and $H^{\flat}=g\circ F^{\flat}$ is standard and preassociative.
\end{proposition}

We now focus on those preassociative functions $F\colon X^*\to Y$ which are unarily quasi-range-idempotent, that is, such that $\ran(F_1)=\ran(F^{\flat})$. It was established in \cite{MarTeh,MarTeh2} that these functions are completely determined by their nullary, unary, and binary parts. Moreover, as the following theorem states, they can be factorized into compositions of associative $\varepsilon$-standard operations with one-to-one unary maps.

\begin{theorem}[{\cite{MarTeh,MarTeh2}}]\label{thm:FactoriAWRI-BPA237111}
Assume AC and let $F\colon X^*\to Y$ be a function. Consider the following assertions.
\begin{enumerate}
\item[(i)] $F$ is preassociative and unarily quasi-range-idempotent.

\item[(ii)] There exists an associative $\varepsilon$-standard operation $H\colon X^*\to X\cup\{\varepsilon\}$ and a one-to-one function $f\colon\ran(H^{\flat})\to Y$ such that $F^{\flat}=f\circ H^{\flat}$.
\end{enumerate}
Then $(i) \Rightarrow (ii)$. If $F$ is standard, then $(ii) \Rightarrow (i)$. Moreover, if condition (ii) holds, then we have $F^{\flat}=F_1\circ H^{\flat}$, $f=F_1|_{\ran(H^{\flat})}$, $f^{-1}\in Q(F_1)$, and we may choose $H^{\flat}=g\circ F^{\flat}$ for any $g\in Q(F_1)$.
\end{theorem}

\begin{remark}\label{rem:fsd68s}
\begin{enumerate}
\item[(a)] If condition (ii) of Theorem~\ref{thm:FactoriAWRI-BPA237111} holds, then by Eq.~(\ref{eq:saf76sf5xx}) we see that $F$ can be computed recursively by
    $$
    F_n(x_1,\ldots, x_n) ~=~ F_2((g\circ F_{n-1})(x_1,\ldots, x_{n-1}),x_n),\qquad n\geqslant 3,
    $$
    where $g\in Q(F_1)$. A similar observation was already made in a more particular setting for the so-called quasi-associative functions; see \cite{Yag87}.

\item[(b)] It is necessary that $F$ be standard for the implication $(ii) \Rightarrow (i)$ to hold in Theorem~\ref{thm:FactoriAWRI-BPA237111}. Indeed, take $a\in\R$ and the function $F\colon\R^*\to\R$ defined by $F(\varepsilon)=a$ and $F^{\flat}(\bfx)=x_1$. Then condition (ii) holds for $H^{\flat}=F^{\flat}$ and $f=\id$. However, $F$ is neither standard nor preassociative since $F(a)=F(\varepsilon)$ and $F(ab)=a\neq b=F(b)$ for every $b\in\R\setminus\{a\}$.
\end{enumerate}
\end{remark}

\section{Unarily idempotizable functions}

In this section we examine the special case of unarily idempotizable standard functions, i.e., unarily quasi-range-idempotent standard functions with one-to-one unary parts.

As far as associative $\varepsilon$-standard operations are concerned, we have the following immediate result.

\begin{proposition}
If $F\colon X^*\to X\cup\{\varepsilon\}$ is an associative $\varepsilon$-standard operation and $F_1$ is one-to-one (hence $F$ is unarily idempotizable), then $F$ is unarily idempotent (i.e., $F_1=\id$).
\end{proposition}

\begin{proof}
Since $F_1\circ F_1=F_1$ we simply have $F_1=F_1^{-1}\circ F_1=\id$.
\end{proof}

We also have the next result, from which we can immediately derive the following corollary.

\begin{theorem}[{\cite{MarTeh,MarTeh2}}]\label{thm:sdf87fs}
Let $F_1\colon X\to X$ and $F_2\colon X^2\to X$ be two operations. Then there exists an associative $\varepsilon$-standard operation $G\colon X^*\to X\cup\{\varepsilon\}$ such that $G_1=F_1$ and $G_2=F_2$ if and only if the following conditions hold:
\begin{enumerate}
\item[(i)] $F_1\circ F_1=F_1$ and $F_1\circ F_2=F_2$,

\item[(ii)] $F_2(x,y)=F_2(F_1(x),y)=F_2(x,F_1(y))$,

\item[(iii)] $F_2$ is associative.
\end{enumerate}
Such an operation $G$ is then uniquely determined by $G_n(x_1,\ldots,x_n)=G_2(G_{n-1}(x_1,\ldots,x_{n-1}),x_n)$ for $n\geqslant 3$.
\end{theorem}

\begin{corollary}
Let $F_1\colon X\to X$ and $F_2\colon X^2\to X$ be two operations. Then there exists an associative and unarily idempotizable $\varepsilon$-standard operation $G\colon X^*\to X\cup\{\varepsilon\}$ such that $G_1=F_1$ and $G_2=F_2$ if and only if $F_1=\id$ and $F_2$ is associative. Such an operation $G$ is then uniquely determined by $G_n(x_1,\ldots,x_n)=G_2(G_{n-1}(x_1,\ldots,x_{n-1}),x_n)$ for $n\geqslant 3$.
\end{corollary}

Regarding preassociative and unarily quasi-range-idempotent standard functions, we have the following results.

\begin{proposition}\label{prop:dsa98as}
Assume AC and let $F\colon X^*\to Y$ be a function. If condition (ii) of Theorem~\ref{thm:FactoriAWRI-BPA237111} holds, then the following assertions are equivalent.
\begin{enumerate}
\item[(i)] $F_1$ is one-to-one,

\item[(ii)] $H_1$ is one-to-one,

\item[(iii)] $H_1=\id$.
\end{enumerate}
\end{proposition}

\begin{proof}
(i) $\Rightarrow$ (iii) $H_1=F_1^{-1}\circ F_1=\id$.

(iii) $\Rightarrow$ (ii) Trivial.

(ii) $\Rightarrow$ (i) $F_1=f\circ H_1$ is one-to-one as a composition of one-to-one functions.
\end{proof}

\begin{corollary}\label{cor:sfa5sfds}
Let $F\colon X^*\to Y$ be a function such that $F_1$ is one-to-one. Consider the following assertions.
\begin{enumerate}
\item[(i)] $F$ is preassociative and unarily quasi-range-idempotent.

\item[(ii)] There is a unique $\varepsilon$-standard operation $H\colon X^*\to X\cup\{\varepsilon\}$ such that $F^{\flat}=F_1\circ H^{\flat}$, namely $H^{\flat}=F_1^{-1}\circ F^{\flat}$. This operation is associative and unarily idempotent.
\end{enumerate}
Then $(i) \Rightarrow (ii)$. If $F$ is standard, then $(ii) \Rightarrow (i)$.
\end{corollary}

\begin{proof}
The proof follows from Theorem~\ref{thm:FactoriAWRI-BPA237111} and Proposition~\ref{prop:dsa98as}. Here AC is not required since the quasi-inverse of $F_1$ is simply an inverse.
\end{proof}

\begin{corollary}\label{cor:sdf87fspa1}
Let $F_1\colon X\to Y$ and $F_2\colon X^2\to Y$ be two functions and suppose that $F_1$ is one-to-one. Then there exists a preassociative and unarily quasi-range-idempotent standard function $G\colon X^*\to Y$ such that $G_1=F_1$ and $G_2=F_2$ if and only if $\ran(F_2)\subseteq\ran(F_1)$ and the function $H_2=F_1^{-1}\circ F_2$ is associative. In this case we have $G^{\flat}=F_1\circ H^{\flat}$, where $H\colon X^*\to X\cup\{\varepsilon\}$ is the unique associative $\varepsilon$-standard operation having $H_1=\id$ and $H_2$ as unary and binary parts, respectively.
\end{corollary}

\begin{proof}
The proof follows from Theorem~4.11 in \cite{MarTeh,MarTeh2} and Proposition~\ref{prop:dsa98as} in this paper.
\end{proof}

The following result is a reformulation of Corollary~\ref{cor:sdf87fspa1}, where $F_2$ is replaced with $H_2=F_1^{-1}\circ F_2$.

\begin{corollary}\label{cor:sdf87fspa}
Let $F_1\colon X\to Y$ and $H_2\colon X^2\to X$ be two functions and suppose $F_1$ is one-to-one. Then there exists a preassociative and unarily quasi-range-idempotent standard function $G\colon X^*\to Y$ such that $G_1=F_1$ and $G_2=F_1\circ H_2$ if and only if $H_2$ is associative. In this case we have $G^{\flat}=F_1\circ H^{\flat}$, where $H\colon X^*\to X\cup\{\varepsilon\}$ is the unique associative $\varepsilon$-standard operation having $H_1=\id$ and $H_2$ as unary and binary parts, respectively.
\end{corollary}

Corollary~\ref{cor:sdf87fspa} shows how the preassociative and unarily quasi-range-idempotent standard functions with one-to-one unary parts can be constructed. Just provide a nullary function $F_0$, a one-to-one unary function $F_1$, and a binary associative function $H_2$. Then $F^{\flat}=F_1\circ H^{\flat}$, where $H$ is the associative $\varepsilon$-standard operation having $H_1=\id$ and $H_2$ as unary and binary parts, respectively.

\section{Axiomatizations of some classes of associative and preassociative functions}

In this section we derive axiomatizations of classes of preassociative functions from certain existing axiomatizations of classes of associative operations. We restrict ourselves to a small number of classes. Further axiomatizations can be derived from known classes of associative operations.

The approach that we use here is the following. Starting from a class of binary associative operations $F\colon X^2\to X$, we identify all the possible associative $\varepsilon$-standard operations $F\colon X^*\to X\cup\{\varepsilon\}$ which extend these binary operations (this reduces to identifying the possible unary parts using Theorem~3.5 in \cite{MarTeh,MarTeh2}). If the unary parts are one-to-one, then we use Corollary~\ref{cor:sfa5sfds}; otherwise we use Theorem~\ref{thm:FactoriAWRI-BPA237111}.

\subsection{Preassociative functions built from Acz\'elian semigroups}
\label{sec:Acz}

Let us recall an axiomatization of the Acz\'elian semigroups due to Acz\'el~\cite{Acz49} (see also \cite{Acz04,CouMar12,CraPal89}).

\begin{proposition}[{\cite{Acz49}}]\label{prop:Aczel2}
Let $I$ be a nontrivial real interval (i.e., nonempty and not a singleton). An operation $H\colon I^2\to I$ is continuous, one-to-one in each argument, and associative if and only if there exists a continuous and strictly
monotonic function $\varphi\colon I\to J$ such that
\begin{equation}\label{eq:s7af5}
H(x,y)=\varphi^{-1}\left(\varphi(x)+\varphi(y)\right),
\end{equation}
where $J$ is a real interval of one of the forms $]{-\infty},b[$, $]{-\infty},b]$, $]a,\infty[$, $[a,\infty[$ or $\R ={]{-\infty},\infty[}$
$(b\leqslant 0\leqslant a)$. For such an operation $H$, the interval $I$ is necessarily open at least on one end. Moreover, $\varphi$ can be chosen to be strictly increasing.
\end{proposition}

According to Theorem~3.5 in \cite{MarTeh,MarTeh2}, every associative $\varepsilon$-standard operation $H\colon I^*\to I\cup\{\varepsilon\}$ whose binary part is of form (\ref{eq:s7af5}) must be unarily idempotent. Indeed, we must have
$$
\varphi^{-1}\left(\varphi(x)+\varphi(y)\right) ~=~ H_2(x,y) ~=~ H_2(H_1(x),y) ~=~ \varphi^{-1}\left(\varphi(H_1(x))+\varphi(y)\right)
$$
and hence $H_1(x)=x$. Thus, there is only one such associative $\varepsilon$-standard operation, which is defined by
$$
H_n(\bfx) ~=~ \varphi^{-1}\left(\varphi(x_1)+\cdots +\varphi(x_n)\right),\qquad n\in\N.
$$
Proposition~\ref{prop:dsa98as} and Corollary~\ref{cor:sfa5sfds} then show how a class of preassociative and unarily idempotizable standard functions can be constructed from $H$.

\begin{theorem}\label{thm:AczelPA}
Let $I$ be a nontrivial real interval (i.e., nonempty and not a singleton). A standard function $F\colon I^*\to\R$ is preassociative and unarily quasi-range-idempotent, and $F_1$ and $F_2$ are continuous and one-to-one in each argument if and only if there exist continuous and strictly monotonic functions $\varphi\colon I\to J$ and $\psi\colon J\to\R$ such that
$$
F_n(\bfx) ~=~ \psi(\varphi(x_1)+\cdots +\varphi(x_n)),\qquad n\in\N,
$$
where $J$ is a real interval of one of the forms $\left]-\infty, b\right[$, $\left]-\infty, b\right]$, $\left]a,\infty\right[$, $\left[a,\infty\right[$ or
$\R = \left]-\infty,\infty\right[$ ($b \leqslant 0 \leqslant a$). For such a function $F$, we have $\psi=F_1\circ\varphi^{-1}$ and $I$ is necessarily open at least on one end. Moreover, $\varphi$ can be chosen to be strictly increasing.
\end{theorem}

\begin{proof}
(Necessity) By Corollary~\ref{cor:sfa5sfds}, the $\varepsilon$-standard operation $H\colon X^*\to X\cup\{\varepsilon\}$ defined by $H^{\flat}=F_1^{-1}\circ F^{\flat}$ is associative. Moreover, $H_2$ is clearly continuous and one-to-one in each argument since so are $F_1^{-1}$ and $F_2$. We then conclude by Proposition~\ref{prop:Aczel2}.

(Sufficiency) By Corollary~\ref{cor:sfa5sfds} and Proposition~\ref{prop:Aczel2}, $F$ is preassociative and unarily quasi-range-idempotent. Moreover, $F_1$ and $F_2$ are continuous and one-to-one in each argument.
\end{proof}

\subsection{Preassociative functions built from t-norms and related operations}

Recall that a \emph{t-norm} (resp.\ \emph{t-conorm}) is an operation $H\colon [0,1]^2\to [0,1]$ which is nondecreasing in each argument, symmetric, associative, and such that $H(1,x)=x$ (resp.\ $H(0,x)=x$) for every $x\in [0,1]$. Also, a \emph{uninorm} is an operation $H\colon [0,1]^2\to [0,1]$ which is nondecreasing in each argument, symmetric, associative, and such that there exists $e\in\left]0,1\right[$ for which $H(e,x)=x$ for every $x\in [0,1]$. For general background see, e.g., \cite{AlsFraSch06,FodRou94,GraMarMesPap09,KlMes05,KleMesPap20,SchSkl83}.

Let us see how t-norms can be used to generate preassociative functions. We first observe that the associative $\varepsilon$-standard operation which extends any t-norm is unique and unarily idempotent; we call such an operation a (\emph{variadic}) \emph{t-norm}. Indeed, from the condition $H(1,x)=x$ it follows that $H(x)=H(H(1,x))=H(1,x)=x$. Using Corollary~\ref{cor:sfa5sfds}, we then obtain the following axiomatization.

\begin{theorem}\label{thm:pretnorm}
Let $F\colon [0,1]^*\to\R$ be a standard function such that $F_1$ is strictly increasing (resp.\ strictly decreasing). Then $F$ is preassociative and unarily quasi-range-idempotent, and $F_2$ is symmetric, nondecreasing (resp.\ nonincreasing) in each argument, and satisfies $F_2(1,x)=F_1(x)$ for every $x\in [0,1]$ if and only if there exist a strictly increasing (resp.\ strictly decreasing) function $f\colon [0,1]\to\R$ and a variadic t-norm $H\colon [0,1]^*\to [0,1]\cup\{\varepsilon\}$ such that $F^{\flat}=f\circ H^{\flat}$. In this case we have $f=F_1$.
\end{theorem}

\begin{proof}
(Necessity) By Corollary~\ref{cor:sfa5sfds}, the $\varepsilon$-standard operation $H\colon [0,1]^*\to [0,1]\cup\{\varepsilon\}$ defined by $H^{\flat}=F_1^{-1}\circ F^{\flat}$ is associative. Moreover, $H_2$ is clearly symmetric, nondecreasing in each argument, and such that $H_2(1,x)=x$. Hence $H$ is a t-norm.

(Sufficiency) By Corollary~\ref{cor:sfa5sfds}, $F$ is preassociative and unarily quasi-range-idemp{\-}otent. Moreover, $F_1$ and $F_2$ clearly satisfy the stated properties.
\end{proof}

If we replace the condition ``$F_2(1,x)=F_1(x)$'' in Theorem~\ref{thm:pretnorm} with ``$F_2(0,x)=F_1(x)$'' (resp.\ ``$F_2(e,x)=F_1(x)$ for some $e\in\left]0,1\right[$''), then the result still holds provided that the t-norm is replaced with a t-conorm (resp.\ a uninorm).

\subsection{Preassociative functions built from Ling's axiomatizations}

The next proposition gives an axiomatization due to Ling \cite{Lin65}; 
see also \cite{Bac86,Mar00}. We remark that this characterization can be easily deduced from previously known results on topological semigroups (see Mostert and Shields \cite{MosShi57}). However, Ling's proof is elementary.

\begin{proposition}[{\cite{Lin65}}]\label{prop:Linga}
Let $[a,b]$ be a real closed interval. An operation $H\colon [a,b]^2\to [a,b]$ is continuous, nondecreasing in each argument, associative, and such that $H(b,x)=x$ for all $x\in [a,b]$ and $H(x,x)<x$ for all $x\in \left]a,b\right[$, if and only if there exists a continuous and strictly decreasing function $\varphi\colon [a,b]\to\left[0,\infty\right[$, with $\varphi(b)=0$, such that
$$
H(x,y) ~=~ \varphi^{-1}(\min\{\varphi(x)+\varphi(y),\varphi(a)\}).
$$
\end{proposition}

Proceeding as in Section~\ref{sec:Acz}, we obtain the following characterization.

\begin{theorem}
Let $[a,b]$ be a real closed interval and let $F\colon [a,b]^*\to\R$ be a standard function such that $F_1$ is strictly increasing (resp.\ strictly decreasing). Then $F$ is unarily quasi-range idempotent and preassociative, and $F_2$ is continuous and nondecreasing (resp.\ nonincreasing) in each argument, $F_2(b,x)=F_1(x)$ for every $x\in [a,b]$, $F_2(x,x)<F_1(x)$ (resp.\ $F_2(x,x)>F_1(x)$) for every $x\in\left]a,b\right[$ if and only if there exist a continuous and strictly decreasing function $\varphi\colon [a,b]\to\left[0,\infty\right[$, with $\varphi(b)=0$, and a strictly decreasing (resp.\ strictly increasing) function $\psi\colon [0,\varphi(a)]\to\R$ such that
$$
F_n(\bfx) ~=~ \psi(\min\{\varphi(x_1)+\cdots +\varphi(x_n),\varphi(a)\}),\qquad n\in\N.
$$
For such a function, we have $\psi=F_1\circ \varphi^{-1}$.
\end{theorem}

\subsection{Preassociative functions built from range-idempotent functions}

Recall that an $\varepsilon$-standard operation $F\colon X^*\to X\cup\{\varepsilon\}$ is said to be \emph{range-idempotent} \cite{CouMar11} if $F(k\Cdot x)=x$ for every $x\in\ran(F^{\flat})$ and every $k\in\N$, where $F(k\Cdot x)$ stands for the unary function $F(x,\ldots,x)$ obtained by repeating $k$ times the variable $x$. Equivalently, $F(k\Cdot F(\bfx))=F(\bfx)$ for every $\bfx\in X^*$ and every $k\in\N$.

We say that a function $F\colon X^*\to Y$ is \emph{invariant by replication} if for every $\bfx\in X^*$ and every $k\in\N$ we have $F(k\Cdot\bfx)=F(\bfx)$, where $F(k\Cdot\bfx)$ stands for the function $F(\bfx,\ldots,\bfx)$ obtained by repeating $k$ times the tuple $\bfx$. More generally, we say that a function $F\colon X^*\to Y$ is \emph{preinvariant by replication} if for every $\bfx,\bfy\in X^*$ and every $k\in\N$ we have that $F(\bfx)=F(\bfy)$ implies $F(k\Cdot\bfx)=F(k\Cdot\bfy)$. Clearly, if a function $F\colon X^*\to Y$ is invariant by replication or preassociative, then it is preinvariant by replication.

Also, if an $\varepsilon$-standard operation $F\colon X^*\to X\cup\{\varepsilon\}$ is unarily range-idempotent and invariant by replication, then it is range-idempotent. Indeed, we simply have $F(k\Cdot F(\bfx))=F(F(\bfx))=F(\bfx)$ for every $\bfx\in X^*$ and every $k\in\N$.

Recall \cite{CouMar11} that an associative $\varepsilon$-standard operation $F\colon X^*\to X\cup\{\varepsilon\}$ is range-idempotent if and only if it is invariant by replication. Moreover, if any of these conditions holds, then we have $F(\bfx,F(\bfx,y,\bfz),\bfz)=F(\bfx,y,\bfz)$ for all $y\in X$ and all $\bfx,\bfz\in X^*$.

\begin{lemma}\label{lemma:7sdf5}
Let $F\colon X^*\to X\cup\{\varepsilon\}$ be an $\varepsilon$-standard operation. The following assertions are equivalent.
\begin{enumerate}
\item[(i)] $F$ is associative and range-idempotent.

\item[(ii)] $F$ is associative and $F(F(x),F(x))=F(x)$ for every $x\in X$.

\item[(iii)] $F$ is preassociative, unarily quasi-range-idempotent, and range-idempotent.

\item[(iv)] $F$ is preassociative, unarily quasi-range-idempotent, and satisfies $F_1\circ F_1=F_1$ and $F(F(x),F(x))=F(x)$ for every $x\in X$.
\end{enumerate}
\end{lemma}

\begin{proof}
(i) $\Rightarrow$ (iii) $\Rightarrow$ (iv) $\Rightarrow$ (ii) Follows from Propositions~\ref{prop:22f1f1f1} and \ref{prop:A-PA1}.

(ii) $\Rightarrow$ (i) Since $F$ is associative, it is unarily quasi-range-idempotent. Thus, we only need to prove that $F(k\Cdot F(x))=F(x)$ for every $x\in X$ and every $k\in\N$. Due to our assumptions, this condition holds for $k=1$ and $k=2$. Now, suppose that it holds for some $k\geqslant 2$. We then have $F((k+1)\Cdot F(x))=F_2(F_k(k\Cdot F(x)),F(x))=F_2(F(x),F(x))=F(x)$ and hence the condition holds for $k+1$.
\end{proof}

The following result yields an axiomatization of a class of associative and range-idempotent $\varepsilon$-standard operations (hence invariant by replication) over bounded chains. If $X$ represents a bounded chain, we denote the classical lattice operations on $X$ by $\wedge$ and $\vee$. Also, the ternary median function on $X$ is the function $\med\colon X^3\to X$ defined by
$$
\med(x,y,z) ~=~ (x\vee y)\wedge (y\vee z)\wedge (z\vee x).
$$

\begin{proposition}\label{prop:sda5f7sfdsa}
Let $H\colon X^*\to X\cup\{\varepsilon\}$ be an $\varepsilon$-standard operation over a bounded chain $X$. Then the following three assertions are equivalent.
\begin{enumerate}
\item[(i)]
\begin{enumerate}
\item[(a)] $H$ is associative,

\item[(b)] $H_2(H_1(x),H_1(x))=H_1(x)$ for every $x\in X$,

\item[(c)] $H_1$ and $H_2$ are nondecreasing in each argument, and

\item[(d)] the sets $H_1(X)=\ran(H_1)$, $H_2(X,z)$, and $H_2(z,X)$ are convex for every $z\in X$.
\end{enumerate}
\item[(ii)]
\begin{enumerate}
\item[(a)] $H$ is associative,

\item[(b)] $H$ is range-idempotent,

\item[(c)] for every $n\in\N$, $H_n$ is nondecreasing in each argument, and

\item[(d)] for every $n\in\N$, the set $H_n(\bfy,X,\bfz)$ is convex for every $(\bfy,\bfz)\in X^{n-1}$.
\end{enumerate}
\item[(iii)] There exist $a,b,c,d\in X$, with $a\leqslant c\wedge d$ and $c\vee d\leqslant b$, such that
$$
H_n(\bfx) ~=~ \med\left(a,(c\wedge x_1)\vee \med\left(\bigwedge_{i=1}^nx_i,c\wedge d,\bigvee_{i=1}^nx_i\right)\vee (d\wedge x_n),b\right),\qquad n\in\N.
$$
\end{enumerate}
\end{proposition}

\begin{proof}
(i) $\Rightarrow$ (ii) Assume that $H\colon X^*\to X\cup\{\varepsilon\}$ satisfies condition (i). By Lemma~\ref{lemma:7sdf5}, $H$ is range-idempotent. Using Eq.~(\ref{eq:saf76sf5}), we immediately see that, for every $n\in\N$, $H_n$ is nondecreasing in each argument.

Now, let us show that, for every $z\in X$ and every convex subset $C$ of $X$, the set $H_2(C,z)$ is convex. Let $t\in X$ such that $H_2(y_0,z)< t< H_2(y_1,z)$ for some $y_0,y_1\in C$. Since $H_2$ is nondecreasing and $H_2(X,z)$ is convex, there is $u\in [y_0,y_1]\subseteq C$ such that $t=H_2(u,z)$. This shows that $H_2(C,z)$ is convex.

We now show by induction on $n$ that $H_n(\bfy,X,\bfz)$ is convex for every $(\bfy,\bfz)\in X^{n-1}$. This condition clearly holds for $n=1$ and $n=2$. Assume that it holds for some $n\geqslant 2$ and let us show that it still holds for $n+1$, that is, $H_{n+1}(\bfy,X,\bfz)$ is convex for every $(\bfy,\bfz)\in X^n$. If $\bfz=\varepsilon$, then since $H_2(z,X)$ is convex so is the set $H_{n+1}(\bfy,X)=H_2(H_n(\bfy),X)$. If $\bfz=(\bfv,z)$, then the set $C=H_n(\bfy,X,\bfv)$ is convex and hence so is the set $H_{n+1}(\bfy,X,\bfv,z)=H_2(C,z)$.

(ii) $\Rightarrow$ (i) Trivial.

(ii) $\Leftrightarrow$ (iii) This equivalence was proved in \cite{CouMar11}.
\end{proof}

\begin{remark}
If we set $c=d$ in any of the operations $H$ described in Proposition~\ref{prop:sda5f7sfdsa}(iii), then its restriction to $[a,b]^*$ is a \emph{$c$-median} \cite{GraMarMesPap09}, that is,
$$
H_n|_{[a,b]^n}(\bfx) ~=~ \med\left(\bigwedge_{i=1}^nx_i,c,\bigvee_{i=1}^nx_i\right).
$$
\end{remark}

We observe that any operation $H\colon X^*\to X\cup\{\varepsilon\}$ satisfying the conditions stated in Proposition~\ref{prop:sda5f7sfdsa} has a unary part of the form $H_1(x)=\med(a,x,b)$, which is not always one-to-one. We will therefore make use of Theorem~\ref{thm:FactoriAWRI-BPA237111} (instead of Corollary~\ref{cor:sfa5sfds}) to derive the following generalization of Proposition~\ref{prop:sda5f7sfdsa} to preassociative functions.

\begin{theorem}\label{thm:sda5f7sfdsa2}
Let $F\colon X^*\to Y$ be a standard function, where $X$ and $Y$ are chains and $X$ is bounded, and let $a,b\in X$ such that $a\leqslant b$. Then the following three assertions are equivalent.
\begin{enumerate}
\item[(i)]
\begin{enumerate}
\item[(a)] $F$ is preassociative and unarily quasi-range-idempotent,

\item[(b)] there exists a strictly increasing function $f\colon [a,b]\to Y$, with a convex range, such that $F_1(x)=(f\circ \med)(a,x,b)$ for all $x\in X$,

\item[(c)] $F_2(x,x)=F_1(x)$ for every $x\in X$,

\item[(d)] $F_2$ is nondecreasing in each argument, and

\item[(e)] the sets $F_2(X,z)$, and $F_2(z,X)$ are convex for every $z\in X$.
\end{enumerate}
\item[(ii)]
\begin{enumerate}
\item[(a)] $F$ is preassociative and unarily quasi-range-idempotent,

\item[(b)] there exists a strictly increasing function $f\colon [a,b]\to Y$, with a convex range, such that $F_1(x)=(f\circ \med)(a,x,b)$ for all $x\in X$,

\item[(c)] for every integer $n\geqslant 2$, we have $F_n(n\Cdot x)=F_1(x)$ for every $x\in X$,

\item[(d)] for every integer $n\geqslant 2$, $F_n$ is nondecreasing in each argument, and

\item[(e)] for every integer $n\geqslant 2$, the set $F_n(\bfy,X,\bfz)$ is convex for every $(\bfy,\bfz)\in X^{n-1}$.
\end{enumerate}
\item[(iii)] There exist $c,d\in [a,b]$ and a strictly increasing function $f\colon [a,b]\to Y$, with a convex range, such that
$$
F_n(\bfx) ~=~ (f\circ\med)\left(a,(c\wedge x_1)\vee \med\left(\bigwedge_{i=1}^nx_i,c\wedge d,\bigvee_{i=1}^nx_i\right)\vee (d\wedge x_n),b\right),\qquad n\in\N.
$$
In this case we have $f=F_1|_{[a,b]}$.
\end{enumerate}
\end{theorem}

\begin{proof}
(i) $\Rightarrow$ (iii) Let $F\colon X^*\to Y$ be a standard function satisfying condition (i) and take $g\in Q(F_1)$ such that $(g\circ F_1)(a)=a$ and $(g\circ F_1)(b)=b$ (this is always possible due to the form of $F_1$ and does not require AC). Thus, $g|_{\ran(F_1)}=f^{-1}$ is strictly increasing. Let $H\colon X^*\to X\cup\{\varepsilon\}$ be the $\varepsilon$-standard operation defined by $H^{\flat}=g\circ F^{\flat}$. By Theorem~\ref{thm:FactoriAWRI-BPA237111}, $H$ is associative and $F^{\flat}=f\circ H^{\flat}$.

Let us show that $H(H(x),H(x))=H(x)$ for every $x\in X$. Using condition (i)(c), we have $H_2(H_1(x),H_1(x))=(g\circ F_2)(H_1(x),H_1(x))=(g\circ F_1)(H_1(x))=(g\circ F_1\circ g\circ F_1)(x)=(g\circ F_1)(x)=H_1(x)$.

The function $H_1(x)=\med(a,x,b)$ is clearly nondecreasing. Since $F_2$ is nondecreasing in each argument, so is $H_2=g\circ F_2$.

The set $H_1(X)=\ran(H_1)=[a,b]$ is convex. Let us show that the set $H_2(X,z)$ is also convex for every $z\in X$. Let $t\in X$ such that $H_2(y_0,z)< t< H_2(y_1,z)$ for some $y_0,y_1\in X$. Since $\ran(H_2)\subseteq\ran(H_1)$, we have $t\in\ran(H_1)=[a,b]$. Therefore, since $f$ is increasing, we have
$$
F_2(y_0,z) ~=~ (f\circ H_2)(y_0,z) ~ \leqslant ~ f(t) ~ \leqslant ~ (f\circ H_2)(y_1,z) ~=~ F_2(y_1,z)
$$
and hence there exists $u\in X$ such that $f(t)=F_2(u,z)$. If follows that $t=(f^{-1}\circ F_2)(u,z)=H_2(u,z)$ and hence that the set $H_2(X,z)$ is convex. We show similarly that $H_2(z,X)$ is convex.

Thus, the operation $H$ satisfies the conditions stated in Proposition~\ref{prop:sda5f7sfdsa} and we have $\ran(H)=\ran(H_1)=[a,d]$.

(iii) $\Rightarrow$ (ii) Combining Theorem~\ref{thm:FactoriAWRI-BPA237111} and Proposition~\ref{prop:sda5f7sfdsa}, we obtain that $F$ is preassociative and unarily quasi-range-idempotent. Also, for every integer $n\geqslant 2$, $F_n$ is nondecreasing in each argument and we have $F_n(n\Cdot x) = (f\circ\med)(a,x,b) = F_1(x)$ for every $x\in X$. Finally, the set $F_n(\bfy,X,\bfz)=(f\circ H_n)(\bfy,X,\bfz)$ is convex for every $(\bfy,\bfz)\in X^{n-1}$ since $f$ is strictly increasing and both sets $\ran(f)$ and $H_n(\bfy,X,\bfz)$ (see Proposition~\ref{prop:sda5f7sfdsa}) are convex.

(ii) $\Rightarrow$ (i) Trivial.
\end{proof}

In the special case where $X$ is a real closed interval and $Y=\R$, the convexity conditions can be replaced with the continuity of the corresponding functions in both Proposition~\ref{prop:sda5f7sfdsa} and Theorem~\ref{thm:sda5f7sfdsa2}. Moreover, every $H_n$ (resp.\ $F_n$) is symmetric if and only if $c=d$. We then have the following corollary.

\begin{corollary}
Let $I$ be a real closed interval and let $[a,b]$ be a closed subinterval of $I$. A standard function $F\colon I^*\to\R$ is preassociative and unarily quasi-range-idempotent, there exists a continuous and strictly increasing function $f\colon [a,b]\to\R$ such that $F_1(x)=(f\circ \med)(a,x,b)$, $F_2$ is continuous, symmetric, nondecreasing in each argument, and satisfies $F_2(x,x)=F_1(x)$ for every $x\in I$ if and only if there exist $c\in [a,b]$ and a continuous and strictly increasing function $f\colon [a,b]\to\R$ such that
$$
F_n(\bfx) ~=~ (f\circ\med)\left(a,\med\left(\bigwedge_{i=1}^nx_i,c,\bigvee_{i=1}^nx_i\right),b\right),\qquad n\in\N.
$$
In this case we have $f=F_1|_{[a,b]}$.
\end{corollary}

\section{Concluding remarks and open problems}

In this paper we have first recalled the concept of preassociativity, a recently-introduced property which naturally generalizes associativity for variadic functions. Then, starting from known axiomatizations of noteworthy classes of associative operations, we have provided characterizations of classes of preassociative functions which are unarily quasi-range-idempotent. We observe that, from among these preassociative functions, the associative ones can be identified by using Proposition~\ref{prop:A-PA1}.

We end this paper by the following interesting questions:

\begin{enumerate}
\item Find interpretations of the preassociativity property in aggregation function theory and/or fuzzy logic.

\item Find new axiomatizations of classes of preassociative functions from existing axiomatizations of classes of associative operations. Classes of associative operations can be found in \cite{AlsFraSch06,Bac86,CouMar11,CouMar12,CraPal89,CzoDre84,Dom12,FodRou94,GraMarMesPap09,KlMes05,KleMesPap20,Lin65,Mar98,Mar00,MarMat11,MosShi57,San02,San05}.
\end{enumerate}

\section*{Acknowledgments}

The authors would like to thank the anonymous reviewers for helpful suggestions. This research is supported by the internal research project F1R-MTH-PUL-12RDO2 of the University of Luxembourg.


\end{document}